\documentclass{article}
\usepackage{amsmath,amssymb,amsthm}

\usepackage{calrsfs}
\usepackage[all]{xy}

\newtheorem{theorem}{Theorem}[section]
\newtheorem{proposition}[theorem]{Proposition}
\newtheorem{corollary}[theorem]{Corollary}

\theoremstyle{definition}

\theoremstyle{remark}
\newtheorem{remark}[theorem]{Remark}
\newtheorem{example}[theorem]{Example}

\newcommand{\cat}{\mathrm{C}}
\newcommand{\mon}{\mathbf{Mon}}
\newcommand{\id}{\mathrm{id}}
\newcommand{\cosk}{\mathbf{Cosk}}

\author{Ivan Yudin
\\ Departamento de Matematica\\ Universidade de Coimbra\\Apartado 3008\\
3001-454 Coimbra\\ Portugal\\
\texttt{yudin@mat.uc.pt}
\thanks{The work is supported by the FCT Grant SFRH/BPD/31788/2006. The
financial support by CMUC and FCT gratefully acknowledged.}
}
\title{The nerve of a crossed module}

\begin{document}
\maketitle
\begin{abstract}
	We give an explicit description for the nerve of crossed module of
	categories.
\end{abstract}
\section{Introduction}

Let $T$ be a topological space. It is said that $T$ has a type $k$ if all the
homotopy groups $\pi_n\left( T \right)$ are zero for $n>k$. It is known that the
categories of groups and of $1$-types are equivalent. In~\cite{nervegroup}
 Eilenberg  and Maclane constructed for every group $G$ a simplicial set
$BG$ such that the topological realization $\left|BG\right|$ of $BG$ is the corresponding
$1$-type. In fact they gave three different description for $BG$ called
homogeneous, non-homogeneous and matrix description. They used these
descriptions to get the explicit chain complex that computes the cohomology
groups of $\left|BG\right|$. This was the born of the homology theory for
algebraic objects. 

It turns out that non-homogeneous description of $BG$ is the most useful one.
This description
was used by Hochschild in~\cite{hochschild} to define the Hochschild complex for
an arbitrary associative algebra $A$ that coincides with the complex
constructed by Eilenberg and Maclane when $A$ is a group algebra.
It also inspired the definition of the nerve of small category and definition of
Barr cohomology. In fact, it is difficult to image the modern mathematics
without non-homogeneous description of $BG$.

In \cite{whitehead} Whitehead showed that $2$-types can be described by crossed
modules of groups. Blakers constructed in~\cite{blakers} for every crossed
module of groups $\left( A,G \right)$ the complex $N^B_*\left( A,G \right)$
whose geometrical realization is the $2$-type corresponding to $\left( A,G
\right)$. In fact he has done this for arbitrary crossed complexes of groups
that describe $k$-type for any $k\in\mathbb{N}$. In the case of $k=1$ his description
coincides with the matrix description of Eilenberg-Maclane for $BG$.

In this article we give an explicit description of a simplicial set $N\left( A,\cat \right)$ for a
crossed monoid $\left( A,\cat \right)$ in terms of certain matrices. This
simplicial set is isomorphic to the one constructed by Blakers in case
$\left( A,\cat \right)$ is a crossed module of groups $\left( A,G \right)$. The difference is that
the elements of $N_k\left( A,\cat \right)$ are described as collections of
elements in $A$ and $G$ without any relations between them, however  the elements
of $N^B_k\left( A,G \right)$ are described as collections of elements in
$A$ and $G$ that should satisfy certain conditions between them. 

The paper is organized as follows. In Section~\ref{simplicial} we recall the
definition of simplicial set and their elementary properties.
Section~\ref{crossed} contains the main result of the paper. Namely, we describe
the simplicial set $N\left( A,\cat \right)$ for an arbitrary crossed monoid
$\left( A,G \right)$. In Theorem~\ref{main1} we prove that $N\left( A,\cat
\right)$ is indeed a simplicial set. 

In Section~\ref{coskeletal} we prove that $N\left( A,\cat \right)$ is $
4$-coskeletal. Moreover, in case $\left( A,\cat \right)$ is a crossed module of
groups it turns out that $N\left( A,\cat \right)$ is $3$-coskeletal.

In Section~\ref{kan} we check that $N\left( A,\cat \right)$ is a Kan simplicial
set if $\left( A,\cat \right)$ is a crossed module of groups. We also check that
the homotopy groups of $\left( A,\cat \right)$ and $N\left( A,\cat \right)$ are
isomorphic in this case. 

In the next version of this paper we shall give a comparison between our
construction and the construction of Blakers~\cite{blakers} and the construction
of Moerdijk and Svensson~\cite{moerdijk}. 

\section{Simplicial set}
\label{simplicial}
For the purpose of this paper a simplicial set is a sequence of sets $X_n$,
$n\ge 0$ with maps $d_j\colon X_n\to X_{n+1}$ and $s_j\colon X_n\to X_{n-1}$,
$0 \le j\le n$
such that for $i<j$:
\begin{align}
	\label{simp1}
	d_jd_k &= d_{k-1} d_j \\  
	\label{simp2}
	d_j s_k& = s_{k-1} d_j \\
	\label{simp3}
	d_j s_j& = \id \\
	\label{simp4}
	d_{j+1} s_j &= \id\\
	\label{simp5}
	d_k s_j& = s_j d_{k-1}  \\
	\label{simp6}
	s_j s_{k-1}& = s_{k }s_j .
\end{align}

The \emph{$n$-truncated simplicial set} is defined as a sequence of sets
$X_0$, \dots, $X_n$ with the maps $d_j\colon X_k\to X_{k-1}$, $s_j\colon X_k\to
X_{k+1}$ for all $k$ and $j$ they have sense, that satisfy the same identities
as same-named maps for a simplicial set. 

We denote the category of simplicial sets by $\triangle^{op} Sets$ and the
category of $n$-truncated simplicial sets by $\triangle^{op}_{n} Sets$. Then we
have an obvious forgetful functor $tr^n\colon \triangle^{op} Sets \to
\triangle_n^{op} Sets$. This functor has a right adjoint $cosk^n\colon
\triangle^{op}_n Sets \to \triangle^{op} Sets$. 
The composition functor $cosk^n tr^n$
will be denoted by $\cosk^n$. 
Thus $\cosk^n$ is a monad on the category of simplicial sets. We say that
$X$ is \emph{$n$-coskeletal} if the unit map $\eta_X\colon X \to \cosk^n X$ is
an isomorphism.

For every simplicial set $X_\bullet$ we define $\bigwedge^n X$ as a simplicial
kernel of the maps $d_j\colon X_{n-1}\to X_{n-2}$, $0\le j\le n-1$. In other
words $\bigwedge^n X$ is a collection of sequences $\left( x_0,\dots,x_n
\right)$, $x_j\in X_{n-1}$, such that  $d_j x_k = d_{k-1}x_j$ for all $0\le
j<k\le n-1$. We have the natural boundary map $b_n\colon X_n\to \bigwedge^n X$ defined by
$$
b_n\colon x\mapsto \left( d_0\left( x \right), \dots, d_n\left( x \right) \right). 
$$
\begin{proposition}
	\label{coskelet}
	Let $X$ be a simplicial set. Then $X$ is $n$-coskeletal if and only if
 for every
	$N>  n$ the map  $b_N$ is a  bijection. 
\end{proposition}
\begin{proof}
	Note that for every $N>n$ the canonical map  $$\cosk^n\left( X \right)
	\to \cosk^{N-1}\left(\cosk^n\left( X \right)\right)$$ is an isomorphism.
	Thus if $X$ is $n$-coskeletal it is also $N-1$-coskeletal. 
	Therefore the maps $X_N\to \cosk^N\left( X \right)_{N}$ are isomorphisms
	for all $N>n$. Int Section~2.1 in~\cite{duskin} it is shown that these
	maps coincide with $b_N$. This shows that the maps $b_N$ are
	isomorphisms for all $N>n$.

	Now suppose that all the maps $b_N$ are isomorphisms. 
	The map $\eta_X\colon X\to \cosk^n X$ is an isomorphism in all degrees
	up to $n$ by definition of the functor $\cosk^n$. We proceed further by
	induction on degree. Suppose we know that $\eta_X\colon X \to \cosk^n X$
	is an isomorphism in all degrees up to $N\ge n$. Therefore the map 
	\begin{align}
		\label{map_eta}
		\tau\colon \bigwedge\nolimits^{N+1}  X \to \bigwedge\nolimits^{N+1}\cosk^n X
	\end{align}
	induced by the $N$-th component of $\eta_X$ is an isomorphism. But now
	the set on the right hand side of \eqref{map_eta} is
	$\left(\cosk^{N}\cosk^n X\right)_{N+1} \cong \left( \cosk^n X
	\right)_{N+1}$. As $n$-th component of $\eta_X$ decomposes into the
	product of $\tau$ and $b_n$ we get that it is an isomorphism.  
\end{proof}
Define the set $\bigwedge^n_l X$ of \emph{$l$-horns in dimension $n$} to be the
collection of $n$-tuples $\left( x_0,\dots, \widehat{x}_l,\dots, x_n \right)$ of
elements in $X_{n-1}$ such that $d_jx_k = d_{k-1} x_j$ for all $0\le j<k \le
n-1$ different from $l$. There are the natural maps
\begin{align*}
b^n_l\colon  X_n & \longrightarrow \bigwedge\nolimits^n_l X \\
x & \mapsto \left( d_0\left( x \right), \dots, \widehat{d_l\left( x
\right)}, \dots, d_n\left( x \right) \right).
\end{align*}
A complex $X$ is said to be \emph{Kan complex} if the maps $b^n_l$ are
surjective for all $0\le l\le n$. 
We define now based homotopy groups $\pi_n\left( X,x \right)$ for a Kan complexes
$X$. We follow to the exposition of~\cite{smirnov} on the pages 27-28.
Let $x\in X_0$. Then all the degenerations $s_{i_n}\dots s_{i_1}\left( x
\right)$ of $x$ in degree $n$ are mutually equal and will be denoted by the same
letter $x$. We define $\pi_n\left( X,x \right)$ to be the set
$$
\left\{\, y\in X_n \,\middle|\,  b^n\left( y \right) = \left( x,\dots,x \right) \right\}
$$
factorized by the equivalence relation
$$
y \sim z \Leftrightarrow \exists w\in X_{n+1} \colon b^{n+1}\left( w
\right) = \left( x,\dots,x, y,z \right).
$$
That $\sim$ is indeed an equivalence relations for a Kan set is shown at the end
of page 27 of \cite{smirnov}. Now we define a multiplication on $\pi_n\left(
X, x\right)$ as follows. Let $\left[ y \right]$, $\left[ z \right]\in
\pi_n\left( X,x \right)$ be equivalence classes containing $y$ and $z$,
respectively. Then the tuple
$$
\left( x,\dots,x, y,\varnothing, z \right)
$$
is an element of $\bigwedge^{n+1}_{n}$. Therefore there is an element $w\in
X_{n+1}$ such that $b^{n+1}_n \left( w \right) = \left( x,\dots,x,y,\varnothing,
z\right)$. We define $[y][z] = \left[d_n \left( w \right)\right]$. Again it is
shown in \cite{smirnov}, that this product is well defined and associative, $[x]$ is the
neutral element, and if $n\ge 2$ the product is commutative. 

There is a connection between coskeletal and Kan conditions for a simplicial
set. To see this we start with
\begin{proposition}
	\label{coskeletal_kan}
	Let $\left( x_0,\dots,\widehat{x}_l,\dots, x_n \right)\in \bigwedge^n_l
	X$. Then
	\begin{equation}
		\label{eq:coskeletal_kan}
	\left(
	y_0,\dots,y_{ n-1} \right)= \left(
	d_{l-1}x_0,\dots,d_{l-1}x_{l-1}, d_lx_{l+1}, \dots, d_lx_n \right)\in
	\bigwedge\nolimits^{n-1}X.
\end{equation}
	\end{proposition}
\begin{proof}
		 Suppose $0\le j< k\le l-1$. Then 
\begin{align*}
	d_ky_j& = d_k\left( d_{l-1}x_j \right) = d_kd_{l-1}x_j  = d_{l-2}d_k x_j
	\\& =
	d_{l-2}d_j
	x_{k+1} = d_j\left( d_{l-1}x_{k+1} \right) = d_j y_{k+1}.
\end{align*}
	For $0\le j\le l-1<k\le n-1$ we get
	\begin{align*}
		d_ky_j &= d_kd_{l-1}x_j = d_{l-1}d_{k+1}x_j = d_{l-1}d_j
		x_{k+2} \\&= d_j d_lx_{k+2} = d_j y_{k+1}.
	\end{align*}

	Finally for $l-1<j<k\le n-1$ we have
	\begin{align*}
		d_k y_j &= d_kd_lx_{j+1} = d_l d_{k+1}x_{j+1} = d_l
		d_{j+1}x_{k+2} \\&= d_j d_{l}x_{k+2} = d_jy_k.
	\end{align*}
\end{proof}
Thus we have a well defined map \label{beta}$\beta^n_l\colon \bigwedge^n_l X \to
\bigwedge^{n-1} X$ given by \eqref{eq:coskeletal_kan}.

As a simple corollary of Proposition~\ref{coskeletal_kan} we get
\begin{corollary}
	\label{criterion}
	Suppose $b_n$ and $b_{n-1}$ are surjections. Then for every $0\le l\le
	n$ the maps $b^n_l$ are surjections.
\end{corollary}
\begin{proof}
	Let $x= \left( x_0,\dots,\widehat{x}_l,\dots,x_n \right)\in
	\bigwedge^n_{l} X$. Then by Proposition~\ref{coskeletal_kan}
	$$\beta^n_lx = \left(
	d_{l-1}x_0,\dots,d_{l-1}x_{l-1}, d_lx_{l+1}, \dots, d_lx_n \right)\in
	\bigwedge\nolimits^{n-1}X.$$ Since $b_{n-1}$ is surjective 
	there is $x_l\in X_{n-1}$ such that $d_jx_l = d_{l-1}x_j$ for $0\le j\le
	l-1$ and $d_jx_l = d_l x_{j+1}$ for $l\le j\le n-1$. Therefore
	$\left( x_0,\dots,x_n \right)\in \bigwedge^n X$ and since $b_n$ is
	surjective there is $z\in X_{n}$ such that $d_j z = x_j$, $0\le j\le n$. 
\end{proof}
\section{Category crossed monoids}
\label{crossed}
Let $\cat$ be a small category. We denote by $\cat_0$ the set of objects and by
$\cat_1$ the set of morphisms of $\cat$. 
 We will write $s\left(\alpha 
\right)$ for the source and $t\left( \alpha \right)$ for the target of the morphism
$\alpha\in \cat_1$. If $F\colon \cat\to \mon$ is a contravariant functor from
$\cat$ to the category of monoids, for $\alpha\in \cat\left( s,t \right)$ and $m\in F\left(
t \right)$ we write $m^\alpha$ for the result of applying
$F\left( \alpha \right)$ 
 to $m$. 

 A \emph{crossed monoid} over $\cat$ is a contravariant functor $A\colon \cat\to
 \mon$ together with a collection of functions $\partial_t\colon A\left( t
 \right)\to \cat\left( t,t \right)$, $t\in \cat_0$, such that 
\begin{align}
	\label{cr1}
\partial_t\left( a \right)& = t\\
	\label{cr2}
	\alpha\partial_s\left( a^\alpha \right)& = \partial_t( a) \alpha\\
	\label{cr3}
	ab& = b a^{\partial_t b} 
\end{align}
for all $s$, $t\in \cat_0$, $\alpha\in \cat\left( s,t \right)$, $a$ , $b\in A\left( t
\right)$. 
We will write $e_x$ for the unit of $A(x)$, $x\in \cat_0$. 

A morphism from a crossed module $\left( A,\cat\right)$ to a crossed module
$\left( B,\widetilde{\cat} \right)$ is a pair $\left( f,F \right)$, where
$F\colon \cat\to
\widetilde{\cat}$ is a functor and $f$ is a collection of homomorphisms $f_x\colon A\left( x
\right)\to B\left( F\left( x \right) \right) $ of monoids such that 
\begin{align}
	\label{eq:mor1}
	f_s\left( a^\alpha \right)&= f_t\left( a \right)^{F\left( \alpha \right)}\\
	F\left( \partial_t\left( a \right) \right)&= \partial_{F\left( t
	\right)}\left( f_t\left( a \right) \right)
\end{align}
for all $s$, $t\in \cat_0$, $\alpha\in \cat\left( s,t \right)$, $a\in A\left( t
\right)$. 
We denote the category of crossed monoids over small categories by $\mathbf{XMon}$. 
Note that $\mathbf{XMon}$ contains a full subcategory $\mathbf{XMod}$ of
\emph{crossed modules} whose objects $\left( A,\cat \right)$ are such that
$\cat$ is a groupoid and $A\left(t  \right)$ is a group for every $t\in \cat_0$.

Now we describe the nerve functor $N\colon \mathbf{XMon}\to
\triangle^{op}\mathbf{Sets}$ into the category of simplicial sets. 
Define $N_0\left( A,\cat \right) = \cat_0$.
 For
$n\ge 1$ we define $N_n\left( A,G \right)$ to be the set of $n\times k$ upper
triangular\footnote{Upper triangular means that  the places in the matrix under the diagonal are empty.} matrices $M = \left( m_{ij} \right)_{i\le j}$ such that 
there is a sequence $x\left( M \right) = \left( x_0\left( M \right),\dots,x_n\left( M \right) \right)$ of objects in $\cat$ such that 
\begin{itemize}
	\item $m_{jj}\in \cat\left( x_{j},x_{j-1} \right)$, $1\le j\le n$;
	\item $m_{ij}\in A\left( x_i \right)$ for $1\le i<j\le n$.\end{itemize}
We will identify  $N_1\left( A,\cat \right)$ with  $\cat_1$. 
We extend function $x$ on $N_0\left( A,cat \right)= \cat_0$ by $x\left( p
\right):= \left( p \right)$. 

Below we will sometimes indicate the empty places with the sign
$\varnothing$.  

Define $s_0\colon N_0\left( A,\cat \right)$ by $s_0\left( p \right) = 1_p$,
$p\in \cat_0$. 
For $n\ge 1$ and $0\le j\le n$ the matrix $M\in N_{n+1}\left( A,\cat \right)$
will be constructed from $M\in N_{n}\left( A,\cat \right)$ as follows
\begin{enumerate}
	\item first insert $e_{x_i\left( M \right)}$ at the $(j+1)$-st place of every row $i$ above the $j+1$-st row;
	\item insert $\left( \varnothing,\dots,\varnothing,1_{x_{j}\left( M \right)},
		e_{x_{j}\left( M \right)}, \dots, e_{x_{j}\left( M \right)}
		\right)$ as  the $j+1$-st row, where $1_{x_{j+1}\left( M
		\right)}$ stay on the $(j+1)$-st place.
	\item shift all elements below $(j+1)$-st row   one position to the
		right.
\end{enumerate}
\begin{example}
	For  $M\in N_3\left( A,\cat \right)$, $j=1$,  and $\left( x_0,x_1,x_2,x_3,x_4 \right) = x\left( M \right)$ we get
	\begin{equation*}
		\xymatrix@C=4em@R=6ex{
		{\left(
		\begin{array}{ccc}
			m_{11} & m_{12} & m_{13}\\
			& m_{22} & m_{23}\\
			&& m_{33}
		\end{array}
		\right)}\ar@{|->}[r]^(.45){\mbox{Step 1}}\ar@{|->}[d]_(.45){s_1} & 
		{\left(
		\begin{array}{cccc}
			m_{11} & e_{x_1} & m_{12} & m_{13}\\
& m_{22} & m_{23}\\
			&& m_{33}
		\end{array}
		\right)}
	\ar@{|->}[d]^(.45){\mbox{Step 2}}\\
	{\left( 
	\begin{array}{cccc}
				m_{11} & e_{x_1} & m_{12} & m_{13}\\
				& 1_{x_1} & e_{x_1} &e_{x_{1}}				\\
&& m_{22} & m_{23}\\
&			&& m_{33}
	\end{array}
	\right)}
	& 
		{\left( 
	\begin{array}{cccc}
				m_{11} & e_{x_1} & m_{12} & m_{13}\\
					& 1_{x_1} & e_{x_1} &e_{x_{1}}				\\
& m_{22} & m_{23}\\
			&& m_{33}
	\end{array}
	\right).} \ar@{|->}[l]_{\mbox{Step 3}} &
}
	\end{equation*}. 
\end{example}
Note that in the case $j=0$ the first step is skipped and in the case $j=n$ the
last step is skipped. 

Now define $d_0\colon N_1\left( A,\cat \right)\to N_0\left( A,\cat \right)$ to
be $s\colon \cat_1\to \cat_0$, and $d_1\colon N_1\left( A,\cat \right)\to
N_0\left( A,\cat \right)$ to be $t\colon \cat_1\to \cat_0$. 
Let $n\ge 2$ and $M\in N_n\left( A,\cat \right)$. We construct the matrix
$d_j\left( M \right) \in N_{n-1}\left( A,\cat \right)$ as follows
\begin{enumerate}
	\item if $j=0$ we just delete the first row;
	\item if $j= n$
		delete the last column;
	\item if $1\le j\le n-1$
		\begin{enumerate}
			\item at every row above the $j$-th row we multiply
				elements at $j$-th and $\left( j+1 \right)$-st
				places;
			\item shift all the elements at $j$-th row and below one
				position to the left;
			\item replace $j$-th and $\left( j+1 \right)$-st rows
				with the row:
				\begin{multline*}
\left(\varnothing ,\dots,\varnothing , m_{jj} \partial\left( m_{j,j+1}
\right)m_{j+1,j+1}, m_{j,j+2}^{\eta_{j+1,j+1}}m_{j+1,j+2}, \dots,\right.\\
\left. m_{j,n}^{\eta_{j+1,n-1}}m_{j+1,n} \right),
\end{multline*}
where
			 \begin{equation}
				 \label{eta}
			 			 \eta_{jk} =
						 m_{j+1,j+1}\partial\left(
						 m_{j+1,j+2}\dots
					 			 m_{j+1,k}\right).
							 			 \end{equation}
		\end{enumerate}
\end{enumerate}
For example
\begin{equation*}
	\xymatrix@R=8ex{
	{\left( 
	\begin{array}{ccccc}
		m_{11} & m_{12} & m_{13} & m_{14} & m_{15}\\
		& m_{22} & m_{23} & m_{24} & m_{25} \\
		&& m_{33} & m_{34} & m_{35} \\
		&&& m_{44} & m_{45}\\
		&&&& m_{55}
	\end{array}
	\right)} \ar@{|->}[d]^{\mbox{Steps (a) and (b)}} \ar@{|->}`[r]`[dd]^{d_2}[dd] & 
	\\ 
	{
	\left( 
	\begin{array}{ccccc}
		m_{11} & m_{12}m_{13} & m_{14} & m_{15}\\
		m_{22} & m_{23} & m_{24} & m_{25} \\
		& m_{33} & m_{34} & m_{35} \\
		&& m_{44} & m_{45}\\
		&&& m_{55}
	\end{array}
	\right)
	} \ar@{|->}[d]^{\mbox{Step (c)}} &  
	\\
	{\left( 
	\begin{array}{cccc}
		m_{11} & m_{12}m_{13} & m_{14} & m_{15}\\
		& m_{22}\partial\left( m_{23} \right)m_{33} &
		m_{24}^{m_{33}} m_{34} & m_{25}^{m_{33}\partial\left(
		m_{34}
		\right)} m_{35} \\
		&& m_{44} & m_{45} \\
		&&& m_{55}
	\end{array}
	\right)}& 
	 }
\end{equation*}
\begin{theorem}
	\label{main1}
	Let $\left( A,\cat \right)$ be a crossed monoid.
	The sequence of sets $N_n\left( A,\cat \right)$ with the maps $s_j$,
	$d_j$ defined above is a simplicial set.
\end{theorem}
\begin{proof}
	We have to check that the maps $d_j$ and $s_j$ satisfy the simplicial
	identities. For a convenience we divide them into two groups. Let
	$M\in N_n\left( A,\cat \right)$. In the first group we
	put the identities
	\begin{align*}
		d_jd_{j+1}\left( M \right)&= d_j^2\left( M \right) & 
		d_js_j\left( M \right) &= M\\
		\\
		s_{j+1}s_j\left( M \right) &= s_j^2\left( M \right) &
		d_{j+1}s_j\left( M \right) &= M
	&
		d_js_{j+1}\left( M \right)& = s_jd_j\left( M \right).\\
	\end{align*}
	The rest of the identities
	\begin{align*}
		d_jd_k\left( M \right) &= d_{k-1}d_j\left( M \right) &
		d_js_k\left( M \right) &= s_{k-1}d_j\left( M \right)
		\\
		s_js_{k-1} \left( M \right) &= s_ks_j\left( M \right) &
		d_ks_j\left( M \right) &= s_jd_{k-1}\left( M \right),
	\end{align*}
	where $j<k-1$, will be in the second group. 

Note that the effect of action of all above maps on the $i$-th row of the matrix
$M$ for $i<j$ is the same as the effect of action of the same named maps on the
nerve of $A\left( x_i\left( M \right) \right)$. Therefore the equality of the
matrices above the $j$-th row follows from the standard description of the nerve
of monoid.

Now the matrices $s_{j+1}s_j\left( M \right) = s_j^2\left( M \right)$ are equal
strictly under the $\left( j+1 \right)$-st as this part is obtained by shifting
the part of $M$ under the $\left( j-1 \right)$-st row two positions in the
south-east direction in both of them. Let $x=x_j\left( M \right)$. The $j$-th row of $s_{j+1}s_j\left( M
\right)$ is obtained from the sequence $\left( \varnothing, \dots, \varnothing,
1_x, e_x,\dots,e_x\right)$ by inserting $e_x$ after $1_x$ and thus coincides
with the $j$-th row of $s_j^2\left( M \right)$. Since $x_{j+1}\left( s_j\left(
M \right) =x
\right)$ the $\left( j+1 \right)$-st row
of $s_{j+1}s_j\left( M \right)$ is the sequence $\left( \varnothing, \dots,
\varnothing, 1_x,e_x,\dots,e_x \right)$ of the appropriate length. The
$\left( j+1 \right)$-st row of $s_js_j\left( M \right)$ is equal to the
$j$-th row of $s_j\left( M \right)$ and thus is the same sequence. This shows
that $s_{j+1}s_j = s_j^2\left( M \right)$. 

Now for the rest of matrices in the first group the part strictly bellow the
$j$-th row is obtained by shifting the elements of $M$ back and forth. It is not
difficult to see that these shifts bring the same-named elements to the same
positions in all four pairs of matrices. 

Similarly the parts strictly below the $j$-th row in matrices of second group
are obtained by applying the map with greater index and moving elements around.
Again the same elements will be in the same places.

Thus we have only to check that the $j$-th rows are equal in every pair of
matrices.  

We start with the matrices of the second group. Thus from now on $k-1>j$. 
In this case the $j$-th row of $d_jd_k\left( M \right)$ is calculated from $j$-th and
$\left( j+1 \right)$-st rows of $d_k\left( M \right)$:
$$
\begin{array}{cccccc}
	m_{j,j}& m_{j,j+1} & \dots & m_{j,k}m_{j,k+1} & \dots & m_{j,n}\\
	& m_{j+1,j+1} & \dots & m_{j+1,k} m_{j+1, k+1} & \dots &
	m_{j+1,n}\end{array}.
$$
Now the sequence of $\eta$'s defined by \eqref{eta} for the $\left( j+1
\right)$-st row of $d_k\left( M \right)$ is
$$
\left( \eta_{j+1,j+1}, \dots, \widehat{\eta}_{j+1,k}, \dots, \eta_{j+1,n}
\right).
$$
Therefore the $j$-th row of $d_jd_k\left( M \right)$ is 
\begin{multline*}
	\left(\varnothing, \dots, \varnothing,  m_{j,j}\partial\left( m_{j,j+1} \right) m_{j+1,j+1},
	m_{j,j+2}^{\eta_{j+1}} m_{j+1,j+2},\dots, \right.\\\left.\left( m_{j,k}m_{j,k+1}
	\right)^{\eta_{j+1,k}} m_{j+1,k}m_{j+1,k+1}, \dots,
	m_{j,n}^{\eta_{j+1,n-1}}m_{j+1,n} \right).
\end{multline*}
Now the $j$-th row of $d_{k-1}d_j\left( M \right)$ is obtained from the
$j$-th row of $d_j\left( M \right)$ by multiplying elements in the $\left( k-1
\right)$-st and $k$-th columns:
\begin{multline*}
		\left( m_{j,j}\partial\left( m_{j,j+1} \right) m_{j+1,j+1},
	m_{j,j+2}^{\eta_{j+1}} m_{j+1,j+2},\dots, \right.\\\left.
	m_{j,k}^{\eta_{j+1,k-1}} m_{j+1,k} 
	m_{j,k+1}^{\eta_{j+1,k}} m_{j+1,k}m_{j+1,k+1}, \dots,
	m_{j,n}^{\eta_{j+1,n-1}}m_{j+1,n} \right).
\end{multline*}
Thus the $j$-th rows of $d_jd_k\left( M \right)$ and $d_{k-1}d_j\left( M
\right)$ are the same outside the $\left( k-1 \right)$-th column, where the most complicated
looking elements are.  By \eqref{cr3} we get
\begin{align}
m_{j,k}^{\eta_{j+1,k-1}}& m_{j+1,k} 
m_{j,k+1}^{\eta_{j+1,k}} m_{j+1,k}m_{j+1,k+1}=\nonumber \\\label{complicated} &= 
m_{j,k}^{\eta_{j+1,k-1}} m_{j+1,k} 
m_{j,k+1}^{\eta_{j+1,k-1}\partial\left( m_{j+1,k} \right)} m_{j+1,k}m_{j+1,k+1}
\\
&= m_{j,k}^{\eta_{j+1,k-1}}m_{j,k+1}^{\eta_{j+1,k-1}}
m_{j+1,k}m_{j+1,k}m_{j+1,k+1}.\nonumber
\end{align}
This shows that $d_jd_k\left( M \right) = d_{k-1}d_j\left( M \right)$. 

Now we consider the pair of matrices $s_js_k\left( M \right)$ and $s_ks_j\left(
M \right)$. Denote $x_j\left( M
	\right)$ by $x$. The $j$-th row of
	$s_js_{k-1}\left( M \right)$ is the sequence $\left( 1_x,e_x,\dots,e_x
	\right)$ of the appropriate length. Now the $j$-th row of $s_ks_j\left(
	M \right)$ is obtained from the similar sequence, which is shorter by
	one element, by inserting this missing element. Thus $s_js_{k-1}\left(
	M \right) = s_ks_j\left( M \right)$. 

	The $j$-th row of $d_js_k\left( M \right)$ is obtained from $j$-th and
	$\left( j+1 \right)$-st rows of $s_k\left( M \right)$:
	$$
		\left(
		\begin{array}{cccccccccc}
		\varnothing& \dots& \varnothing& m_{j,j}& m_{j,j+1}&
		m_{j,j+1 }& \dots& e_{x_j}& \dots& m_{j,n}\\
		\varnothing & \dots & \varnothing & \varnothing & m_{j+1,j+1} &
		 m_{j+1,j+2}& \dots & e_{x_{j+1} } & \dots & m_{j+1,n}
	\end{array}	
		\right),
$$
where $e$'s are in the $\left( k+1 \right)$-st column. 
Since $\partial\left( e_{x_{j+1}} \right) = 1_{x_{j+1}}$ it is immediate that
the corresponding sequence of $\eta$'s has the form 
$$
\eta_{j+1,j+1}, \dots, \eta_{j+1,k-1}, \eta_{j+1,k}, \eta_{j+1,k},
\eta_{j+1,k+1}, \dots, \eta_{j+1,n},
$$
that is it is obtained from the sequence of $\eta$'s for $M$ by duplicating
$\eta_{j+1,k}$. Since $e_{x_j}^{\eta_{j+1,k}} = e_{x_j}$ we see that the
$j$-th row of $d_js_k\left( M \right)$ can be obtained from the $j$-th row of
$d_j\left( M \right)$ by inserting $e_{x_j}$ at place $k$. Thus the $j$-th row
of $d_js_k\left( M \right)$ is equal to the $j$-th row of $s_{k-1}d_j\left( M
\right)$. 

Further the $j$-row of $s_jd_{k-1}\left( M \right)$ is a sequence of appropriate
length
$$
\left( \varnothing,\dots, \varnothing, 1_x,e_x,\dots,e_x  \right),
$$
where $x=x_j\left( M \right)$. The $j$-th row of $d_ks_j\left( M \right)$
is obtained from the one element longer sequence by multiplying two neighboring
$e_x$. As $e_x^2 = e_x$ we get that $d_ks_j\left( M \right) = s_jd_{k-1}\left(
M \right)$. 

It is left to consider the equalities in the first group. 
First  we will show that the $j$-th rows of
$d_jd_{j+1}\left( M \right)$ and $d_{j}d_j\left( M \right)$ are the same. 
First we consider the most left elements of these rows. For $d_jd_{j+1}\left(
M \right)$ it is equal to
$$
m_{j,j} \partial\left( m_{j,j+1}m_{j,j+2} \right) \left( m_{j+1,j+1}
\partial\left( m_{j+1,j+2} \right) m_{j+2,j+2} \right)
$$
and for $d_j^2\left( M \right)$:
$$
\left( m_{j,j}\partial\left( m_{j,j+1} \right) m_{j+1,j+1} \right)
\partial\left( m_{j,j+2}^{m_{j+1,j+1}} m_{j+1,j+2} \right) m_{j+2,j+2}.
$$
These two elements are equal since 
$$
\partial\left(m_{j,j+2}\right) m_{j+1,j+1} = m_{j+1,j+1} \partial\left(
m_{j,j+2}^{m_{j+1,j+1}} \right)
$$
by \eqref{cr2}. 
Now let $l>j$. We will compute the element at the place $\left( j,l \right)$ in
$d_jd_{j+1}\left( M \right)$ and $d_j^2\left( M \right)$. First note that the
sequence of $\eta$'s for the $\left( j+1 \right)$-st row of $d_j\left( M
\right)$ coincide with the sequence of $\eta$'s of the 
$\left( j+2 \right)$-nd row of $M$. Taking to the account shift of columns on two
positions to the left the element of $d_j^2\left( M
\right)$ at the place $\left( j,l \right)$ is
\begin{equation}
	\label{eq14}
	\left( m_{j,l+2}^{\eta_{j+1,l+1}}m_{j+1,l+2} \right)^{\eta_{j+2,l+1}}
	m_{j+2,l+2}.
\end{equation}
	To compute the corresponding element in $d_jd_{j+1}\left( M \right)$ we have to
find
\begin{align*}
	\eta_{j+1,l}\left( d_{j+1}\left( M \right) \right)& = d_{j+1}\left( M
	\right)_{j+1,j+1} \partial\left( d_{j+1}\left( M \right)_{j+1,j+2}\dots
	d_{j+1}\left( M \right)_{j+1,l}\right) \\
	&=  m_{j+1,j+1}\partial\left( m_{j+1,j+2} \right)
	m_{j+2,j+2}\\&\phantom{=} \times  \partial\left(
	m_{j+1,j+3}^{\eta_{j+2,j+2}}m_{j+2,j+3}\dots
	m_{j+1,l}^{\eta_{j+2,l}}m_{j+2,l+1} \right).
\end{align*}
Now iterating \eqref{complicated} we can write the product under the $\partial$
as
\begin{align*}
	m_{j+1,j+3}^{\eta_{j+2,j+2}}&m_{j+1,j+4}^{\eta_{j+2,j+2}}\dots
	m_{j+1,l}^{\eta_{j+2,j+2}}m_{j+2,j+3}\dots m_{j+2,l+1} =\\&=  \left(
	m_{j+1,j+3}\dots m_{j+1,l+1}
	\right)^{m_{j+2,j+2}} m_{j+2,j+3}\dots m_{j+2,l+1}.
\end{align*}
Since by \eqref{cr2}
\begin{align*}	
	m_{j+2,j+2}\partial&\left(  \left(
	m_{j+1,j+3}\dots m_{j+1,l+1}
	\right)^{m_{j+2,j+2}} \right) =\\&=\partial\left(  
	m_{j+1,j+3}\dots m_{j+1,l+1}
	\right) m_{j+2,j+2} 
	\end{align*}
	we get
	\begin{align*}
		\eta_{j+1,l}\left( d_{j+1}\left( M \right) \right) &=
		m_{j+1,j+1}\partial\left( m_{j+1,j+2} 	m_{j+1,j+3}\dots
		m_{j+1,l+1}
		\right)\\&\phantom{=} \times
		m_{j+2,j+2}\partial\left(m_{j+2,j+3}\dots m_{j+2,l+1}\right) \\
		&=\eta_{j+1,l+1}\eta_{j+2,l+1}. 
	\end{align*}
	Therefore the $\left( j,l \right)$-th element of $d_jd_{j+1}\left( M
	\right)$ is
	\begin{align*}
		d_{j+1}\left( M \right)_{j,l+1}^{\eta_{j+1,l}\eta_{j+2,l}}
		d_{j+1}\left( M \right)_{j+1,l} &= 
		m_{j,l+2}^{\eta_{j+1,l}\eta_{j+2,l}}
		m_{j+1,l+2}^{\eta_{j+2,l+1}} m_{j+2,l+2},	
	\end{align*}
	which is equal to \eqref{eq14}.
	Therefore $d_jd_{j+1}\left( M \right) = d_j^2\left( M \right)$. 

Now  the $j$-th row of $d_js_j\left( M \right)$ is obtained from 
the $j$-th and $\left( j+1 \right)$-st rows of $s_j\left( M \right)$:
$$
\left( 
\begin{array}{ccccccccccc}
	\varnothing & \dots & \varnothing & m_{j,j} & e_{x} & m_{j,j+1} &
	\dots & m_{j,n} \\
	\varnothing & \dots & \varnothing & \varnothing & 1_{x} &
	e_{x} & \dots & e_{x}
\end{array}
\right),
$$
where $x= x_j\left( M \right)$. 
We see that the corresponding sequence of $\eta$'s consist from $1_{x}$
repeated the required number of times. Now
\begin{align*}
	m_{j,j} \partial\left( e_{x} \right) 1_{x}& = m_{j,j}\\
	m_{j,l}^{1_{x}} e_{x} &  =  m_{j,l}  & l&\ge j+1.
\end{align*}
Therefore $d_js_j\left( M \right) = M$. 

The $j$-th row of $d_j s_{j-1}\left( M \right)$ is obtained from the $j$-th and
$\left( j+1 \right)$-st rows of $s_{j-1}\left( M \right)$:
$$
\left( 
\begin{array}{ccccccccccc}
	\varnothing& \dots & \varnothing & 1_{x} & e_x & e_x & \dots & e_x\\
	\varnothing & \dots & \varnothing & \varnothing & m_{j,j} & 
	m_{j,j+1} & \dots & m_{j,n}
\end{array}
\right), 
$$
where $x = x_{j-1}\left( M \right)$. The required sequence of $\eta$'s is the
$j$-th sequence of $\eta$'s for $M$. Now
\begin{align*}
	1_x \partial\left( e_x \right) m_{j,j} & = m_{j,j}\\
	e_x^{\eta_{j,l}} m_{j,l} & = e_{x_{j}} m_{j,l} = m_{j,l} \mbox{ for
	$l>j$.  } 
\end{align*}
Therefore $d_js_{j-1}\left( M \right) = M$. 

Finally we consider the $j$-th row of $d_js_{j+1}\left( M \right)$ and
$s_jd_j\left( M \right)$.
The $j$-th row of the second matrix is obtained from the $j$-th row of
$d_j\left( M \right)$ by inserting $e_x$, $x= x_{j}\left( d_j\left( M
\right) \right) = x_{j+1}\left( M \right)$, at the place $j+1$. The $j$-th row
of $d_js_{j+1}\left( M \right)$ is obtained from the $j$-th and $\left( j+1
\right)$-st rows of $s_{j+1}\left( M \right)$:
	$$
		\left(
		\begin{array}{cccccccccc}
		\varnothing& \dots& \varnothing& m_{j,j}& m_{j,j+1}&
	e_{x_j}& 	m_{j,j+1 }& \dots& m_{j,n}\\
		\varnothing & \dots & \varnothing & \varnothing & m_{j+1,j+1} &
		 e_{x_{j+1} } &  m_{j+1,j+2}& \dots & m_{j+1,n}
	\end{array}	
		\right)
$$
We see that the corresponding sequence of $\eta$'s is obtained from the
$(j+1)$-st sequence of $\eta$'s for $M$ by repeating $\eta_{j+1,j+1}$ twice. It
is straightforward not that the $j$-th row of $d_js_{j+1}\left( M \right)$ is
obtained from the $j$-th row of $d_j\left( M \right)$ by inserting
$e_{x_{j+1}}$ at the place $j+1$. Thus $d_js_{j+1}\left( M \right) =
s_jd_j\left( M \right)$. 
\end{proof}
\section{Coskeletal property }
\label{coskeletal}
In this section we investigate coskeletality of $N\left( A,\cat \right)$ for a
given crossed monoid $\left( A,\cat \right)$. 
For every $n\ge 2$ we denote by $\widetilde{N}_n\left( A,\cat \right)$ the set
of triples $\left( M^0,M^n,m \right)$, where $M^0$, $M^n\in N_{n-1}\left( A,\cat
\right)$, $m\in A\left(s\left( m^0_{11}\right) \right)$ are such that
$d_{n-1} M^0 = d_0 M^n$. We have an obvious map
\begin{align*}
	{\lambda}_n \colon N_n\left( A,\cat \right) &\to \widetilde{N}_n\left(
	A,\cat \right)\\
	M &\mapsto \left( d_0M, d_n M, m_{1n} \right).
\end{align*}
The map is a bijection and we will denote the inverse of $\lambda_n$ by
$\mu_n$. The following picture explains how to construct $\mu_n\left( M^0,M^n,m
\right)\in N_n\left( A,\cat \right)$ for $\left( M^0,M^n, m \right)\in
\widetilde{N}_n\left( A,\cat \right)$: 
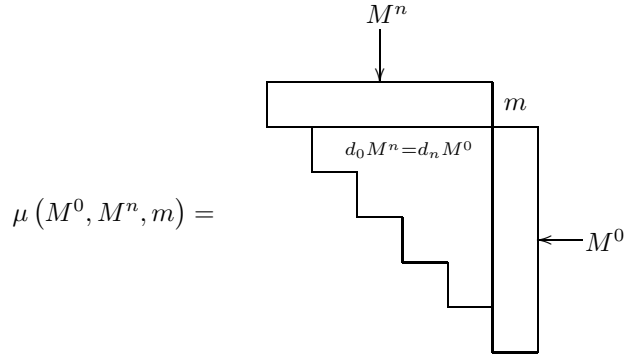
$$
\raisebox{-5em}{$\mu\left( M^0,M^n,m \right) =$\phantom{a;dj}     } \xy (25,-9)*{ {}_{d_0M^n = d_nM^0}};(39,-3)*{m};(6,0);(36,0)**\dir{-};(36,-6)**\dir{-};(42,-6)**\dir{-};
(42,-36)**\dir{-};(36,-36)**\dir{-};(36,-30)**\dir{-};(30,-30)**\dir{-};
(30,-24)**\dir{-};(24,-24)**\dir{-};(24,-18)**\dir{-};(18,-18)**\dir{-};
(18,-12)**\dir{-};(12,-12)**\dir{-};(12,-6)**\dir{-};(6,-6)**\dir{-};
(6,0)**\dir{-};(12,-6);(36,-6)**\dir{-};(36,-30)**\dir{-};
(51,-21)*{M^0};(22,9)*{M^n};\ar@{->}(21,0);(21,7);\ar@{->}(48,-21);(42,-21)
\endxy
$$
Now we investigate the effect of applying $d_j$ to $\mu_n\left( M^0,M^n,m
\right)$. For $j=0$ and $j=n$ we have by definition
\begin{align*}
	d_0\mu_n\left( M^0,M^n,m \right) &= M^0\\
	d_n \mu_n\left( M^0,M^n,m \right) & = M^n.
\end{align*}
Now for $1\le j\le n-1$:
\begin{align*}
	d_0d_j\mu_n\left( M^0,M^n,m \right) & = d_{j-1}d_0 \mu\left( M^0,M^n,m
	\right) = d_{j-1}M^0\\
	d_{n-1}d_j\mu_n\left( M^0,M^n,m \right) & = d_jd_{n}\mu\left( M^0,M^n,m
	\right) = d_jM^n.
\end{align*}
Therefore $d_j\mu_n\left( M^0,M^n,m \right) = \mu_{n-1}\left( d_{j-1}M^0,d_j
M^n,m' \right)$, where $m'$ is the element at the north-east corner of
$d_j\left( M^0,M^n,m \right)$. If $2\le j\le n-2$, then $m'=m$. For $j=1$,
$n-1$ it looks more complicated. Namely, for $j=1$ we get
\begin{align}
	\label{eq:mprime1}
	m' = m^{m^0_{11}\partial\left( m^0_{12}\dots m^0_{1,n-2} \right)}
	m^0_{1,n-1}
\end{align}
	and for $j=n-1$
	\begin{align}
		\label{eq:mprime2}
		m' = m^n_{1,n-1} m, 
	\end{align}
		where $m^0_{i,j}$ and $m^n_{i,j}$ are the entries of $M^0$ and $M^n$,
respectively.
\begin{theorem}
	\label{coskeletal1}
	Let $\left( A,\cat \right)$ be a crossed monoid. Then the simplicial set
	$N\left( A,\cat \right)$ is $4$-coskeletal.
\end{theorem}
\begin{proof}
	Let $n\ge 5$. We have to check that $b^n\colon N_n\left( A,\cat
	\right) \to \bigwedge^n N\left( A,\cat \right)$ is a bijection. Define
	the map $\nu_n\colon \bigwedge^n N\left( A,\cat \right)\to
	\widetilde{N}_n\left( A,\cat \right)$ by 
	\begin{align}	
		\label{eq:nu}
	\nu_n\colon 	\left( M^0,\dots,M^n \right) \mapsto \left( M^0,M^n,
	m^2_{1,n-1}
	\right),
	\end{align}
	where $m^2_{1,n-1}$ is the element of $M^2$ at the upper-right corner.
	We get a commutative triangle
	$$
	\xymatrix{N_n\left( A,\cat \right)\ar[r]^{b_n} \ar[rd]_{\lambda_n} &
	\bigwedge^n N\left( A,\cat \right)\ar[d]^{\nu_n}  \\
	&\widetilde{N}_n\left( A,\cat  \right).}
	$$
	Since $\lambda_n$ is a bijection it follows that $b_n$ is injective. Now
	for $\left( M^0,\dots,M^n \right)\in \bigwedge^n N\left( A,\cat
	\right)$ we define $M = \mu_n\nu_n\left( M^0,\dots,M^n \right)$. We
	claim that $b_n\left( M \right) = \left( M^0,\dots,M^n \right)$. In fact
	\begin{align}
		\nonumber
		d_0 M & = d_0\mu_n\left( M^0,M^n, m^2_{1,n-1} \right) = M^0\\
\nonumber
		d_n M & = d_n \mu_n\left( M^0,M^n, m^2_{1,n-1} \right) = M^n\\
		d_j M & = d_j \mu_n\left( M^0,M^n, m^2_{1,n-1} \right) =
		\mu_{n-1}\left( d_{j-1}M^0, d_jM^n, m' \right)\nonumber\\& = \mu_{n-1}
		\left( d_0M^j, d_{n-1}M^j,m' \right).\label{eq:mprime3}
	\end{align}
	If $2\le j\le n-2$, then $m'=m^2_{1,n-1}$. If $j=2$ then
	$$
	d_2M = \mu_{n-1}\left( d_0M^2, d_{n-1}M^{2}, m^2_{1,n-1} \right) =
	\mu_{n-1}\lambda_{n-1}M^2 = M^2. 
	$$
	If $3\le j\le n-2$, then the element of $d_{j-1}M^2$ at the right-upper
	corner is the same as for $M^2$, and similarly for $d_2M^{j-1}$ and
	$M^{j-1}$. As  
	$d_jM^2 = d_2M^{j-1}$ we get $m^2_{1,n-1} = m^j_{n-1}$ and therefore
	$$
	d_j\left( M \right) = \mu_{n-1}\left( d_0M^j,d_{n-1}M^j, m^j_{1,n-1}
	\right) = M^j.
	$$
	For $j=n-1$ the element $m'$ in \eqref{eq:mprime3} is
	$m^n_{1,n-1}m^2_{1,n-1}$ by \eqref{eq:mprime2}.
Now 
\begin{align*}
	m^n_{1,n-1} m^2_{1,n-1} & = \left( d_2M^n \right)_{1,n-2}m^2_{1,n-1}\\
	& = \left( d_{n-1}M^2 \right)_{1,n-2} m^2_{1,n-1} \\
	&= m^2_{1,n-2}m^2_{1,n-1}\\
	& = \left( d_{n-2}M^2 \right)_{1,n-2} = \left( d_2M^{n-1}
	\right)_{1,n-2} = m^{n-1}_{1,n-1}.
\end{align*}
Note that in the first step we used $n-1\ge 4$ which is equivalent to our assumption
$n\ge 5$.
Combining with \eqref{eq:mprime3} we get
\begin{align*}
	d_{n-1}M = \mu_{n-1}\left( d_0M^{n-1},d_{n-1}M^{n-1},m^{n-1}_{1,n-1}
	\right) = M^{n-1}.
\end{align*}

	For $j=1$ the element $m'$ in \eqref{eq:mprime3} is given by
	\eqref{eq:mprime1}: 	
	\begin{align}	
		\label{eq:mprime4}
		m' = \left({m^2_{1,n-1}}\right)^{m^0_{11}\partial\left(
		m^0_{12}\dots m^0_{1,n-2} \right)}
		m^0_{1,n-1}.
	\end{align}
	We have to show that this product is equal to $m^1_{1,n-1}$. We have
	\begin{align*}
		m^1_{1,n-1}&\stackrel{n-1\ge 4}{=\!=\!=\!=} \left( d_2M^1 \right)_{1,n-2}
		= \left( d_1M^3 \right)_{1,n-2}\\
		\\&= \left( m^3_{1,n-1} \right)^{m^3_{22}\partial\left(
		m^3_{23},\dots m^3_{2,n-2}
		\right)}m^3_{2,n-1}.
	\end{align*}
	This formula already looks similar to \eqref{eq:mprime4}. It is only
	left to identify the elements in both formulas. 
	For $2\le i\le n-1$ we have
	\begin{align*}
		m^3_{2,i} = \left( d_0 M^3 \right)_{1,i-1} = \left( d_2M^0
		\right)_{1,i-1} = 
		\begin{cases}
			m^0_{11} &,i=2\\
			m^0_{12}m^0_{13}&,i=3\\
			m^0_{1i} &, i\ge 4.
		\end{cases}
	\end{align*}
In particular
$$
m^3_{22}\partial\left( m^3_{23}\dots m^3_{2,n-2} \right) =
m^0_{11}\partial\left( m^0_{12}\dots m^0_{1,n-2} \right). 
$$
Thus it is left to show $m^3_{1,n-1} = m^2_{1,n-1}$. This follows from
$$
m^3_{1,n-1} = \left( d_2M^3 \right)_{1,n-2} = \left( d_2M^2 \right)_{1,n-2} =
m^2_{1,n-1}.
$$
Finally $d_{n-1}M= \mu_{n-1}\left( d_0M^{n-1},d_{n-1}M^{n-1},m^{n-1}_{1,n-1}
\right) = M^{n-1}$.
\end{proof}

\begin{theorem}
	\label{coskeletal2}
	Let $\left( A,\cat \right)$ be a crossed monoid such that 
	\begin{itemize}
		\item for every object $t\in\cat$ the monoid $A\left( t
			\right)$ has left and right cancellation properties;
		\item for every morphism $\gamma \in \cat$ the map $a\mapsto
			a^{\gamma}$ from $A\left( t\left( \gamma \right)
			\right)$ to $A\left( s\left( \gamma \right) \right)$ is
			injecive.
	\end{itemize}
	  Then
	$N\left( A,\cat \right)$ is $3$-coskeletal
\end{theorem}
\begin{remark}
	Note that crossed modules satisfy the conditions of the theorem. 
\end{remark}
\begin{proof}
	We already saw in Theorem~\ref{coskeletal1} that $N\left( A,\cat
	\right)$ is $ 4$-coskeletal. Therefore it is enough to show that
	$b_4\colon N_4\left( A,\cat \right)\to \bigwedge^4 N\left( A,\cat
	\right)$ is a bijection. We define the map $\nu_4\colon
	\bigwedge^4N\left( A,\cat \right)\to \widetilde{N}_4\left( A,\cat
	\right)$ by \eqref{eq:nu}. Then $\lambda_4= \nu_4 b_4$ is a bijection.
	Therefore $b_4$ is injective. For $\left( M^0,\dots,M^4 \right)$ we
	define $M= \mu_4\nu_4\left( M^0,\dots,M^4 \right)$. In the same way as
	in the proof of Theorem~\eqref{coskeletal1} we get $d_0 M = M^0$,
	$d_4M= M^4$ and $d_2M^2 = M^2$. Now by \eqref{eq:mprime3} and
	\eqref{eq:mprime2} we get
	$$
	d_3M = \mu_3\left( d_0M^3,d_1M^3,m^4_{13}m^2_{13} \right).
	$$
	To get $d_3M = M^3$ we have to show that $m^4_{13}m^2_{13} =
	m^3_{13}$.
	We have the  following equalities
	\begin{align}
		m^3_{12}m^3_{13} & = \left( d_2M^3 \right)_{12} = \left( d_2M^2
		\right)_{12} = m^2_{12}m^2_{13}\label{eq:case431}\\
		m^3_{12}& = \left( d_3M^3 \right)_{12} = \left( d_3M^4
		\right)_{12} = m^4_{12}\label{eq:case432}\\
		m^2_{12}& = \left( d_3M^2 \right)_{12} = \left( d_2M^4
		\right)_{12} = m^4_{12} m^4_{13}.\label{eq:case433}
	\end{align}
	Therefore
	$$
	m^3_{12}m^3_{13} \stackrel{\eqref{eq:case431}}{\Relbar\!\Relbar}
	m^2_{12}m^2_{13} \stackrel{\eqref{eq:case433}}{\Relbar\!\Relbar}
	m^4_{12}m^4_{13}m^2_{13}  \stackrel{\eqref{eq:case432}}{\Relbar\!\Relbar}
	m^3_{12}m^4_{13}m^2_{13}
	$$
	and by left cancellation for $A\left( x_1\left( M \right) \right)$ we
	get
	$m^3_{13} = m^4_{13}m^2_{13}$ as required.
	
	Now by \eqref{eq:mprime3} and \eqref{eq:mprime1} we get $$d_1M =
	\mu_{n-1}\left( d_0M^1,d_3M^1, \left( m^2_{13}
	\right)^{m^0_{11}\partial\left( m^0_{12} \right)} m^0_{13} \right).$$
	To prove $d_1M= M^1$ it is enough to check that $m^1_{13}= \left( m^2_{13}
	\right)^{m^0_{11}\partial\left( m^0_{12} \right)} m^0_{13}$.	
	We have the equalities
	\begin{align}
		\label{eq:case411}
		\left( m^1_{13} \right)^{m^1_{22} } m^1_{23} & = \left( d_1M^1
		\right)_{12} = \left( d_1M^2 \right) = \left( m^2_{13}
		\right)^{m^2_{22}}m^2_{23}\\
		\label{eq:case412}		m^2_{23}& = \left( d_0M^2 \right)_{12} = \left( d_1M^0
		\right)_{12} = \left( m^0_{13} \right)^{m^0_{22}}m^0_{23}\\
		\label{eq:case413}
		m^1_{23} &= \left( d_0M^1 \right)_{12} = \left( d_0M^0
		\right)_{12} = m^0_{23}\\
		\label{eq:case414}
		m^2_{22} &= \left( d_0M^2 \right)_{11} = \left( d_1M^0
		\right)_{11} = m^0_{11}\partial\left( m^0_{12}
		\right)m^0_{22}\\
		\label{eq:case415}		m^1_{22} &= \left( d_0M^1 \right)_{11} = \left( d_0M^0
		\right)_{11} = m^0_{22}.
	\end{align}
	We get
	\begin{align*}	
		\left( m^1_{13} \right)^{m^0_{22}} m^0_{23}
	&	\stackrel{\eqref{eq:case415},\eqref{eq:case413}}{\Relbar\!\Relbar\!\Relbar\!\Relbar}
		\left( m^1_{13} \right)^{m^1_{22}} m^1_{23}
		\stackrel{\eqref{eq:case411}}{\Relbar\!\Relbar}
		\left( m^2_{13} \right)^{m^2_{22}} m^2_{23}
	\\[3ex]&	\stackrel{\eqref{eq:case414},\eqref{eq:case412}}{\Relbar\!\Relbar\!\Relbar\!\Relbar}
		\left( m^2_{13} \right)^{m^0_{11}\partial\left( m^0_{12} \right)
		m^0_{22}}
		\left( m^0_{13} \right)^{m^0_{22}}m^0_{23} 
\\[2ex]&	=
	\left( \left( m^2_{13} \right)^{m^0_{11}\partial\left( m^0_{12}
	\right)}m^0_{13} \right)^{m^0_{22}} m^0_{23}.
\end{align*}
Now from the right cancellation property for $A\left( x_1\left(M \right)
\right)$ and injectivity of the action of $\cat$ we obtain
		 $$m^1_{13} = \left(
		m^2_{13} \right)^{m^0_{11}\partial\left( m^0_{12}
		\right)}m^0_{13}$$ as required.
\end{proof}

\section{Kan property}
\label{kan}
Recall that the nerve $N\left( \cat \right)$ of a category $\cat$ is a Kan
simplicial set if and only if $\cat$ is a groupoid. In this section we prove
that the nerve $N\left( A,\cat \right)$ of a crossed monoid $\left( A,\cat
\right)$ is a Kan complex if and only if $\left( A,\cat \right)$ is a crossed
module.

Suppose $\left( A,\cat \right)$ is crossed module. Then by Proposition~\ref{coskeletal2} the set $N\left( A,\cat \right)$ is
	$3$-coskeletal. Therefore for $n\ge 5$ by Corollary~\ref{criterion}
	the maps $b^n_j$ are sujective. Now $N\left( \cat \right)$ can be
	embedded into $N\left( A,\cat \right)$ by putting appropriate units over
	the diagonal. At levels $0$ and $1$ this embedding is a bijection.
	As $N\left( \cat \right)$ is a Kan complex we get that the Kan condition
	holds for $N\left( A,\cat \right)$  at degrees $0$ and $1$. 
	Moreover, $N\left( \cat \right)\hookrightarrow N\left( A,\cat \right)$
	induces the  isomorphisms between sets $\bigwedge^2_j N\left( \cat
	\right)$ and $\bigwedge^2_j N\left( A,\cat \right)$, $0\le j\le 2$. As $N_2\left(  \cat
	\right)$ is a subset of $N_2\left(A,\cat  \right)$ and the restriction
	of $b^2_j\colon N_2\left( A,\cat \right)\to \bigwedge^2N\left( A,\cat
	\right)$ to $N_2\left( \cat \right)$ coincide with $b^2_j\colon N_2\left( \cat \right)\to
	\bigwedge^2_j N_2\left( \cat \right)$ we get the Kan property at
level $2$. 
Thus only the sujectivity of maps $b^3_j\colon N_3\left( A,\cat \right)\to
\bigwedge^3 N\left( A,\cat \right)$, $0\le j\le 3$, and $b^4_j\colon N_4\left(
A,\cat \right)\to \bigwedge^4 N\left( A,\cat \right)$, $0\le j\le 4$, should be
checked in order to show that $N\left( A,\cat \right)$ is a Kan simplicial set. 

To deal with this problem we start by description of the image of $b_3\colon N_3\left( A,\cat \right)\to \bigwedge^3
N\left( A,\cat \right)$.
\begin{proposition}
	\label{image}
	Let $\left( A,\cat \right)$ be a crossed module.
	Then $\left( M^0,M^1,M^2,M^3 \right)$ from $\bigwedge^3 N\left( A,\cat
	\right)$ lies in the image of 
	$b^3$ if and only if
	\begin{equation}	
		\label{eq:image}
		\left( m^3_{12} \right)^{m^3_{22}} m^1_{12} =
		\left( m^2_{12} \right)^{m^3_{22}} m^0_{12}.
	\end{equation}
	\end{proposition}
\begin{proof}
	The only if part is true for an arbitrary crossed monoid $\left( A,\cat
	\right)$. In fact, let $M\in N_3\left( A,\cat \right)$, then
\begin{multline}	
	b_3\left( M \right) = \left( 
	\left( 
	\begin{array}{cc}
		m_{22} & m_{23}\\
		& m_{33}
	\end{array}
	\right), 
	\left( 
	\begin{array}{cc}
		m_{11}\partial\left( m_{12} \right) m_{22} & \left( m_{13}
		\right)^{m_{22}} m_{23}\\
		& m_{33}
	\end{array}
	\right),\right.\\ \left.
	\left( 
	\begin{array}{cc}
		m_{11} & m_{12}m_{13}\\
		& m_{22}\partial\left( m_{23} \right) m_{33}
	\end{array}
	\right),
	\left( 
	\begin{array}{cc}
		m_{11} & m_{12} \\
		& m_{22}
	\end{array}
	\right)
	\right).
\end{multline}	
Therefore
\begin{align*}
	\left( \left( d_3M \right)_{12} \right)^{\left( d_3M \right)_{22}}
	\left( d_1M
	\right)_{12} &  = \left( m_{12} \right)^{m_{22}} \left( m_{13}
	\right)^{\left(
	m_{22}
	\right)} m_{23}\\[2ex]& = \left( m_{12}m_{13} \right)^{m_{22}} m_{23} = \left(
	\left(
	d_2M \right)_{12}
	\right)^{\left( d_3M \right)_{22}} \left( d_0M \right)_{12}.
\end{align*}
	Now suppose that $\left( A,\cat \right)$ 
	is a crossed module and $\left( M^0,\dots,M^3 \right)\in \bigwedge^3 N\left(
	A,\cat \right)$ satisfies \eqref{eq:image}. We define $M:=\mu_3\left(
	M^0,M^3, \left( m^3_{12} \right)^{-1}m^2_{12}
	\right)$. Then $d_0M = M^0$ and $d_3M = M^3$. Moreover, by
	\eqref{eq:mprime3} and \eqref{eq:mprime1}
	$$
	d_1M = \mu_2\left( d_0M^1,d_2M^1, \left( \left( m^3_{12}
	\right)^{-1} m^2_{12} \right)^{m^0_{11}}m^0_{12} \right)
	$$
	Since $m^0_{11} = d_2M^0 = d_0M^3 = m^3_{22}$ we get

	$$
	\left( \left( m^3_{12}
	\right)^{-1} m^2_{12} \right)^{m^0_{11}}m^0_{12} = 
\left( \left( m^3_{12}
\right)^{-1} m^2_{12} \right)^{m^3_{22}}m^0_{12}
	\stackrel{\eqref{eq:image}}{\Relbar\!\Relbar} 
	m^1_{11}
	$$
	and therefore $d_1M = \mu_2\left( d_0M^1,d_2M^1, m^1_{12} \right) =
	M^1$. Now by \eqref{eq:mprime3} and \eqref{eq:mprime2}
	$$
	d_2M = \mu_2\left(d_0M^2,d_2M^2, m^3_{12}\left( m^3_{12}
	\right)^{-1} m^2_{12} \right) = \mu_2\left( d_0M^2,d_2M^2,m^2_{12}
	\right) = M^2.
	$$
\end{proof}
Now we can handle Kan property at level $3$ of $N\left( A,\cat \right)$. 
\begin{proposition}
	Let $\left( A,\cat \right)$ be a crossed module. Then for $0\le j\le 3$
	the maps $b^3_j\colon N\left( A,\cat
	\right)\to \bigwedge^3_j N\left( A,\cat \right)$ are sujective. 
\end{proposition}
\begin{proof}
	For every $0\le j\le 3$ and $M^*\in \bigwedge^3_j N\left( A,\cat \right)$ we will
	construct $M^j\in N_2\left( A,\cat \right)$ that extends $M^*$ to
	$\left( M^0,M^1,M^2,M^3 \right)\in \mathrm{Im}\left( b^3 \right)$.  The
	diagonal elements of $M^j$ are determined from the equalities
	\begin{align*}
		m^j_{11} = d_2 M^j = 
\begin{cases}
	d_j M^3 & j<3\\
	d_2 M^2 & j=3
\end{cases} = 
\begin{cases}
	m^3_{22} & j=0 \\
	m^3_{11}\partial\left( m^3_{12} \right)m^3_{22} & j=1\\
	m^3_{11} & j=2\\
	m^2_{11} & j=1.
\end{cases}
	\end{align*}
	\begin{align*}
		m^j_{22} = d_0 M^j = 
		\begin{cases}
			d_0 M^1 & j=0\\
			d_{j-1} M^0 & j>0 
		\end{cases} = 
		\begin{cases}
			m^1_{22} & j=0\\
			m^0_{22} & j=1\\
			m^0_{11}\partial\left( m^0_{12} \right) m^0_{22} & j=2\\
			m^0_{11} & j=3.
		\end{cases}
	\end{align*}
	The element at the right upper corner of $M^j$ is uniquely determined
	from \eqref{eq:image}. The care should be taken for $j=3$: in this case
	we replace $m^3_{22}$ in \eqref{eq:image} by $m^0_{11}$. We
	automatically get that 
	\begin{align*}
	d_0M^j = 
	\begin{cases}
		d_0 M^1 & j=0\\
		d_{j-1} M^0 & j>0
	\end{cases} && 
	d_2M^j= 
	\begin{cases}
		d_j M^3 & j<3\\
		d_2 M^2 & j=3.
	\end{cases}
	\end{align*}
	Thus we have only to check that
	\begin{align*}
		d_1M^j = 
		\begin{cases}
			d_j M^2 & j\le 1\\
			d_{j-1} M^1 & j \ge 2.
		\end{cases}
	\end{align*}
	Below is the required computation. For $j=0$ we have:
	\begin{align*}
		 && d_1M^0 & = m^0_{11}\partial\left( m^0_{12} \right)
		m^0_{22} =
		m^3_{22}
	\partial\left(
	\left(
	\left( m^2_{12} \right)^{-1}
		m^3_{12}
		\right)^{m^3_{22}}
		m^1_{12} \right)
		m^1_{22} \\
		&&& = \partial\left( \left( m^2_{12} \right)^{-1} \right)
		\partial\left( m^3_{12} \right) m^3_{22}\,\, \partial\!\left(
		m^1_{12}
		\right) m^1_{22} 
		\\&&& =  \partial\left( \left( m^2_{12} \right)^{-1} \right)
		\left( m^3_{11} \right)^{-1} \left( d_1M^3 \right)
	 \partial\!\left(
		m^1_{12}
		\right) m^1_{22} 
		\\&&&=  \partial\left( \left( m^2_{12} \right)^{-1} \right)
		\left( d_2M^3 \right)^{-1} \left( d_2M^1 \right)
\partial\!\left(
		m^1_{12}
		\right) m^1_{22} 
\\&&& =   \partial\left( \left( m^2_{12} \right)^{-1} \right)
\left( d_2M^2 \right)^{-1} m^1_{11}\partial\left( m^1_{12} \right)
m^1_{22} 
\\&&& =   \partial\left( \left( m^2_{12} \right)^{-1} \right)
\left( m^2_{11} \right)^{-1} \left( d_1M^1 \right)  = \partial\left( \left( m^2_{12} \right)^{-1} \right)
\left( m^2_{11} \right)^{-1} \left( d_1M^2 \right)\\&&& = m^2_{22}= d_0 M^2. 
	\end{align*}
	For $j=1$
	we get
	\begin{align*}
		d_1M^1 &= m^1_{11}\partial\left( m^1_{12} \right) m^1_{22} =
		m^3_{11}\partial\left( m^3_{12} \right)m^3_{22}
		\partial\left( \left( \left( m^3_{12}
		\right)^{-1}m^2_{12}\right)^{m^3_{22}} m^0_{12} \right) m^0_{22}
		\\& = m^3_{11}\partial\left( m^3_{12}\left(
		m^3_{12}
		\right)^{-1} m^2_{12}\right) m^3_{22} \partial\left(m^0_{12}  \right)  m^0_{22}
		\\&  = m^2_{11} \partial\left( m^2_{12} \right) m^0_{22}\partial\left(
		m^0_{12} \right)m^0_{22} 
		=  m^2_{11} \partial\left( m^2_{12} \right) \left( d_1M^0\right) 
		\\&= 
 m^2_{11} \partial\left( m^2_{12} \right) \left( d_0M^2 \right) =
 m^2_{11} \partial\left( m^2_{12} \right) m^2_{22} = d_1M^2.
	\end{align*}
	For $j=2$ we replace $m^3_{22}$ in \eqref{eq:image} by $m^0_{11}$ and obtain
	\begin{align*}
		d_1M^2& = m^2_{11}\partial\left( m^2_{12} \right) m^2_{22}\\& = 
		m^3_{11} \partial\left( m^3_{12} \left( m^1_{12}\left(
		m^0_{12}
		\right)^{-1} \right)^{\left( m^0_{11} \right)^{-1}} \right) m^0_{11}\partial\left(
		m^0_{12}
		\right) m^0_{22} \\
		&=
		m^3_{11} \partial\left( m^3_{12} \right) m^0_{11} \partial\left(
		m^1_{12} \left( m^0_{12} \right)^{-1} \right) \partial\left(
		m^0_{12} \right) m^0_{22} 
		\\& = m^3_{11} \partial\left( m^3_{12} \right) m^3_{11}
		\partial\left( m^1_{12} \right) m^0_{22} 
		= \left( d_1M^3 \right) \partial\left( m^1_{12} \right)
		m^1_{22} \\&= \left( d_0 M^1 \right) \partial\left( m^1_{12}
		\right) m^1_{22} = m^1_{11}\partial\left( m^1_{12} \right)
		m^1_{22} = d_1M^1.
	\end{align*}
	Finally for $j=3$
	we get
	\begin{align*}
		d_1M^3 &= m^3_{11}\partial\left( m^3_{12} \right) m^3_{22}
		= m^2_{11} \partial\left(m^2_{12} \left( m^0_{12} \left(
		m^1_{12}
		\right)^{-1} \right)^{\left( m^0_{11} \right)^{-1}}  \right) m^0_{11}
		\\&= m^2_{11}\partial\left( m^2_{12} \right) m^0_{11}
		\partial\left( m^0_{12}\left( m^1_{12} \right)^{-1} \right)
		\\&= \left( d_1M^2 \right)\left( m^2_{22} \right)^{-1} \left(
		d_1M^0 \right)\left( m^0_{22} \right)^{-1} \partial\left(\left(
		m^1_{12}
		\right)^{-1}\right)
		\\&= \left( d_1M^1 \right)\left( d_0M^2 \right)^{-1} \left(
		d_0M^2
		\right) \left( d_0M^0 \right)^{-1} \partial\left( \left(
		m^1_{12}
		\right)^{-1}
		\right)
		\\&= m^1_{11} \partial\left( m^1_{12} \right) m^1_{22}\left(
		m^1_{22} \right)^{-1} \partial\left( \left(
		m^1_{12} \right)^{-1}
		\right) = m^1_{11} = d_2M^1. 
	\end{align*}
\end{proof}
Now we check the Kan condition at the level $ 4$. 
\begin{proposition}
	Let $\left( A,\cat \right)$ be a crossed module. Then for all $0\le j\le
	4$ the map  $b^4_j\colon N_4\left( A,\cat
	\right)\to \bigwedge^4_j N\left( A,\cat \right)$ is surjective.  
\end{proposition}
\begin{proof}
	We know by Proposition~\ref{coskeletal2} that $b_4\colon N_4\left( A,\cat
	\right)\to \bigwedge^4 N\left( A,\cat \right)$ is surjective. Thus if we
	show that for every $0\le j\le 4$ any $M^*\in \bigwedge^4_j N\left(
	A,\cat
	\right)$ can be extended by $M^j\in N_3\left( A,\cat \right)$ to an element of $\bigwedge^4 N\left( A,\cat
	\right)$, the proposition will be proved. 
	The existence of such $M^j$ is equivalent to  $\beta^4_j M^*\in
	\mathrm{Im}\left( b^3 \right)$, where $\beta^4_j\colon \bigwedge^4_j
	N\left( A,\cat \right)\to \bigwedge^3 N\left( A,\cat \right)$ is defined
	on page~\pageref{beta}. Therefore by Proposition~\ref{image} we have to check
	that \eqref{eq:image} holds for $\beta^4_j M^*$.  

	Before doing this let us introduce some notation. We define the elements
	\begin{align*}
		m_{22}&  := m^0_{11} = m^3_{22} = m^4_{22}\\
		m_{23} & := m^0_{12} = m^4_{23}\\
		m_{33} & := m^0_{22} = m^1_{22} = m^4_{33}. 
	\end{align*}
	This definitions should be understand in a way that left hand side
	element is defined to be any of available element in $M^*$ on the right
	hand side. Moreover, if more then one element on the right hand side is
	available then all choices give the same result. The last assertion
	follows from $M^*\in \bigwedge^4_j N\left(A,\cat  \right)$. 

	Now we define the matrix $W = \left( w_{st} \right)_{s,t=0}^4$ to be
	$$
	\left( 
	\begin{array}{ccccc}
		\varnothing & m^1_{23} & m^2_{23} & m^3_{23} & m^4_{23}\\[2ex]
		m^0_{23} & \varnothing & \left( m^2_{13}
		\vphantom{\sum\limits_{1}^{2}}
		\right)^{m_{22}\partial\left( m_{23} \right)m_{33}}\!\!\!\!\!
		m^2_{23} & \left( m^3_{13} 	\vphantom{\sum\limits_{1}^{2}}
\right)^{m_{22}} \!\!\!\!\!m^3_{23} & \left(
		m^4_{13} 	\vphantom{\sum\limits_{1}^{2}}
\right)^{m_{22}} \!\!\!\!\!m^4_{23} \\[2ex]
		\left( m^0_{13} 	\vphantom{\sum\limits_{1}^{2}}
\right)^{m_{33}} \!\!\!\!\!m^0_{23} & \left(
		m^1_{13}	\vphantom{\sum\limits_{1}^{2}}

		\right)^{m_{33}} \!\!\!\!\!m^1_{23} & \varnothing & m^3_{12}m^3_{13} &
		m^4_{12}m^4_{13} \\[2ex]
		m^0_{12}m^0_{13} & m^1_{12}m^1_{13} & m^2_{12} m^2_{13} &
		\varnothing & m^4_{12}\\[2ex]
		m^0_{12} & m^1_{12} & m^2_{12} &m^3_{12} &\varnothing
	\end{array} 
	\right) 
	$$
or in other terms
$$
w_{st} = 
\begin{cases}
	\left(d_{s-1}M^t\right)_{12} & s<t\\
	\left(d_s M^t\right)_{12} & t>s.
\end{cases}
$$
For a given $j$ and $M^*\in \bigwedge^4_j N\left( A,\cat \right)$ only the
elements outside of $j$-th column of $W$ are defined. Moreover, the $j$-th row
of $W$ gives the upper-corner elements of the matrices $\beta^4_j M^*$ and the
relation between them equivalent to \eqref{eq:image} can be read off from the
$j$-th column. It follows
from $M^*\in \bigwedge^4_j N\left( A,\cat \right)$ that if we remove $j$-th
column and $j$-th row from $W$ then the resulting matrix is symmetric. We will
use this fact in the computations bellow. 

Now for $j=0$ we have to check that $w_{02}w_{01}^{-1} = \left(
w_{04}^{-1}
w_{03} \right)^{m_{33}}$. 
We have
\begin{align*}
	w_{01} & = \left( w_{31}^{-1} w_{41} \right)^{m_{33}} w_{21} & 	w_{03} & = \left( w_{23}^{-1} w_{43} \right)^{m_{22}} w_{13}\\
	w_{02} & = \left( w_{32}^{-1} w_{42} \right)^{m_{22} \partial\left(
	m_{23}
	\right)m_{33}} w_{12} & 	w_{04} & = \left( w_{24}^{-1} w_{34}
	\right)^{m_{22}} w_{14}.
\end{align*}
Moreover $m_{23} = m^4_{23} = w_{04}$ and for any $a\in A\left( t \right)$ holds
$a^{\partial\left( w_{04} \right)} = w_{04}^{-1} a w_{04}$. 
Since $W$ is symmetric we get
\begin{align*}
w_{04}^{-1} w_{03}&  = w_{14}^{-1} \left( w_{34}^{-1} w_{24}
	w_{23}^{-1} w_{43} \right)^{m_{22}} w_{13}\\
	w_{02}w_{01}^{-1}& = \left(w_{14}^{-1}\left( w_{34}^{-1}
	w_{24}w_{32}^{-1} w_{42} w_{24}^{-1} w_{34} \right)^{m_{22}}
	w_{14}\right)^{m_{33}} w_{12} w_{21}^{-1} \left( w_{41}^{-1}
	w_{31} \right)^{m_{33}}\\& = \left( w_{14}^{-1} \left( w_{34}^{-1}
	w_{24} w_{32}^{-1} w_{34}\right)^{m_{22}} w_{31}  \right)^{m_{33}}
	= \left( w_{04}^{-1} w_{03} \right)^{m_{33}}. 	
	\end{align*}
For $j=1$ we have to check that $w_{12}w_{10}^{-1}= \left(
w_{14}^{-1}w_{13}
\right)^{m_{33}} $. We have
\begin{align*}
	w_{10} & = \left( w_{30}^{-1} w_{40} \right)^{m_{33}}
	w_{20}  & w_{13} &= \left( w_{43}^{-1} w_{23} \right)^{m_{22}}
w_{03}\\
	w_{12} &= \left( w_{42}^{-1} w_{32}
	\right)^{m_{22}\partial\left( m_{23} \right) m_{33}}
	w_{02} & w_{14} &= \left( w_{34}^{-1} w_{24} \right)^{m_{22}}w_{04}.
\end{align*}
Therefore taking into account that $m_{23} =
w_{04}$ we get
\begin{align*}
	w_{14}^{-1} w_{13} &= w_{04}^{-1} \left(
	w_{24}^{-1} w_{34} w_{43}^{-1} w_{23}
	\right)^{m_{22}} w_{03} = w_{04}^{-1} \left(
	w_{24}^{-1} w_{23}
	\right)^{m_{22}} w_{03}\\
	w_{12}w_{10}^{-1}& = \left(w_{04}^{-1} \left(
	w_{42}^{-1} w_{32}	
	\right)^{m_{22}} w_{04} \right)^{m_{33}}
	w_{02} w_{20}^{-1} \left( w_{40}^{-1}
	w_{30} \right)^{m_{33}}\\& = 
	\left( w_{04}^{-1} \left( w_{42}^{-1}
	w_{32} \right)^{m_{22}} w_{30}
	\right)^{m_{33}} = \left(
	w_{14}^{-1}w_{13} \right)^{m_{33}}. 
\end{align*}
For $j=2$ we have to check that $w_{21}w_{20}^{-1} = \left(
w_{24}^{-1}w_{23} \right)^{m_{22}\partial\left( m_{23} \right)m_{33}}$. 
We have 
\begin{align*}
	w_{20} &= \left( w_{40}^{-1} w_{30} \right)^{m_{33}} w_{10} & 
	w_{23}^{m_{22}} &= w_{43}^{m_{22}} w_{13} w_{03}^{-1} \\
	w_{21} &= \left( w_{41}^{-1} w_{31} \right)^{m_{33}} w_{01} &
	w_{24}^{m_{22}} &= w_{34}^{m_{22}} w_{14} w_{04}^{-1}.
\end{align*}
Therefore
\begin{align*}
	\left( w_{24}^{-1}w_{23} \right)^{m_{22}\partial\left( m_{23} \right)}
	&= 
	w_{04}^{-1} w_{04}w_{14}^{-1} \left( w_{34}^{-1} w_{43}
	\right)^{m_{22}} w_{13}w_{03}^{-1} w_{04}\\
	&= w_{14}^{-1} w_{13} w_{03}^{-1} w_{04}\\
	w_{21}w_{20}^{-1}& = \left( w_{41}^{-1} w_{31}  \right)^{m_{33}}
	w_{01} w_{10}^{-1} \left( w_{30}^{-1} w_{40} \right)^{m_{33}} 
	\\&= \left( w_{41}^{-1} w_{31} w_{30}^{-1} w_{04} \right)^{m_{33}} =
	\left( w_{24}^{-1}w_{23} \right)^{m_{22}\partial\left( m_{23}
	\right)m_{33}}. 
\end{align*}
For $j=3$ we have to check that $w_{31}w_{30}^{-1} = \left( w_{34}^{-1}
w_{32} \right)^{m_{22}}$. 
We have
\begin{align*}
	w_{30}^{m_{33}} & = w_{40}^{m_{33}} w_{20}w_{10}^{-1} &
	w_{32}^{m_{22}\partial\left( m_{23} \right)m_{33}} &= 
	w_{42}^{m_{22}\partial\left( m_{23}  \right)m_{33}}
	w_{12}w_{02}^{-1}\\
	w_{31}^{m_{33} } &= w_{41}^{m_{33}} w_{21} w_{01}^{-1} &
	w_{34}^{m_{22}} &= w_{24}^{m_{22}} w_{04} w_{14}^{-1}. 
\end{align*}
Therefore 
\begin{align*}
	\left( w_{31}w_{30}^{-1} \right)^{\partial\left( w_{04} \right)m_{33}}
	&= \left( w_{04}^{-1}w_{41} \right)^{m_{33}} w_{21} w_{01}^{-1}
	w_{10} w_{20}^{-1} \left( w_{40}^{-1}w_{04} \right)^{m_{33}}\\
	&= \left( w_{04}^{-1}w_{41} \right)^{m_{33}} w_{21} w_{20}^{-1}\\
	\left( w_{34}^{-1}w_{32} \right)^{m_{22}\partial\left( w_{04} \right)
	m_{33}} &= \left( w_{04}^{-1} w_{14} w_{04}^{-1} w_{24}^{m_{22}}
	w_{04} \right)^{m_{33}} \left( w_{04}^{-1} w_{42}^{m_{22}} w_{04}
	\right)^{m_{33}} w_{12} w_{02}^{-1}\\&= 
	\left( w_{04}^{-1} w_{14} \right)^{m_{33}} w_{12}w_{02}^{-1} = \left(
	w_{31}w_{30}^{-1} \right)^{\partial\left( w_{04} \right) m_{33}}.  
\end{align*}
And now the required equality follows from the invertibility of action of
$\cat$
on $A$. 

For $j=4$ we have to check that $w_{41}w_{40}^{-1} = \left( w_{43}^{-1}
w_{42} \right)^{m_{22}}$. We have
\begin{align*}
	w_{40}^{m_{33}} &= w_{30}^{m_{33}} w_{10}w_{20}^{-1} & 
	w_{42}^{m_{22}\partial\left( m_{23} \right)m_{33}} &=
	w_{32}^{m_{22}\partial\left( m_{23} \right)m_{33}} w_{02}
	w_{12}^{-1}\\
	w_{41}^{m_{33}} & = w_{31}^{m_{33}} w_{01} w_{21}^{-1} &
	w_{43}^{m_{22}} & = w_{23}^{m_{22}} w_{03} w_{13}^{-1}. 
\end{align*}
Note that this time we can not use $m_{23} = w_{04}$, instead we will use
$m_{23} = w_{40}$. We get
\begin{align*}
	\left( w_{41} w_{40}^{-1} \right)^{\partial\left( w_{40}\right)m_{33} }
	&= \left( w_{40}^{-1} w_{41} \right)^{m_{33}} =w_{20} w_{10}^{-1} \left(
	w_{30}^{-1} w_{31} \right)^{m_{33}} w_{01} w_{21}^{-1}\\
	\left( w_{43}^{-1} w_{42} \right)^{m_{22}\partial\left( w_{40}
	\right)m_{33}} & = \left( w_{40}^{-1} w_{13} w_{03}^{-1}
	\left( w_{23}^{-1} \right)^{m_{22}} w_{40}
	\right)^{m_{33}}\\&\phantom{=(}\times  \left(
	w_{40}^{-1} w_{32}^{m_{22}} w_{40}
	\right)^{m_{33}} w_{02} w_{12}^{-1} 
	\\& = \left( w_{40}^{-1} \right)^{m_{33}} \left( w_{13}w_{03}^{-1}
	\right)^{m_{33}} w_{40}^{m_{33}} w_{02} w_{12}^{-1}
	\\&= w_{20} w_{10}^{-1} \left( w_{30}^{-1} w_{13} w_{03}^{-1}
	w_{30}  \right)^{m_{33}}w_{10} w_{20}^{-1} w_{02} w_{12}^{-1}
	\\&= w_{20} w_{10}^{-1}\left( w_{30}^{-1} w_{13} \right)^{m_{33}}
	w_{10} w_{12}^{-1}
\end{align*}
and the required equality follows from the invertibility of action of $\cat$ on
$A$. 
\end{proof}
Now we can compute homotopy groups of $N\left( A,\cat \right)$ for a crossed
module $\left( A,\cat \right)$. Let $t\in \cat_0$. Then $\pi\left( N\left(
A,\cat
\right), t
\right)$ is given by the classes $\left[ g \right]$ of elements $g\in \cat_1$
such that $s\left( g \right)=t\left( g \right)=t$, that is $g\in \cat\left( t,t
\right)$. Two elements $g_1$, $g_2\in \cat\left( t,t \right)$ belong to the same
class if and only if there is an element $M\in N_2\left( A,\cat \right)$ such
that $b^2\left( M \right) = \left( 1_t, g_1,g_2 \right)$. This implies
$m_{11} = 1_t$ and $m_{22} = g_2$. Therefore $g_1 = \partial\left( m_{12}
\right)m_{22} = \partial\left( m_{12} \right)g_2 $. As the element $m_{12}\in
A\left( t \right)$ can be chosen arbitrary we see that $g_1$ and $g_2$ are in
the same class if and only if $g_1 \mathrm{Im}\partial_t = g_2
\mathrm{Im}\partial_t$. Note that $\mathrm{Im}\partial_t$ is a normal subgroup
of $\cat\left( t,t \right)$ as for all $g\in \cat\left( t,t \right)$ and
$a\in A\left( t \right)$ we have $g^{-1}\partial\left( a \right)g =
\partial\left( a^g \right)$.  Thus $\pi_1$ can be identified with the quotient
group $\left.\raisebox{0.3ex}{$ \cat\left( t,t
\right)$}\middle/\!\raisebox{-0.3ex}{$ \partial\left( A\left( t \right) \right)$}\right.
$ as a set. Now we show that the composition law on $\pi_1\left( N\left( A,\cat
\right),t
\right)$ coincides with the composition law of $\left.\raisebox{0.3ex}{
$\cat\left( t,t \right)$}\middle/\!\raisebox{-0.3ex}{$ \partial\left( A\left( t
\right)
\right)$}\right. $. Let $\left[ g_1 \right]$, $\left[ g_2 \right]\in \pi_1\left(
N\left( A,\cat \right),t \right)$.  Then 
$$
M := \left( 
\begin{array}{cc}
	g_1 & e_t\\
	& g_2
\end{array}
\right)\in N_2\left( A,\cat \right)
$$
is the preimage of $\left( g_1,\varnothing, g_2 \right)\in \bigwedge^2_1 N\left(
A,\cat \right)$ under $b^2_2$. Therefore $[g_1][g_2] = \left[d_2\left( M \right)  \right] =
\left[ g_1g_2 \right]$. 

Now we compute $\pi_2 := \pi_2\left( N\left( A,\cat \right), t \right)$. The
elements of $\pi_2$ are the classes $\left[ M \right]$ of elements $M\in
N_2\left( A,\cat \right)$ such that $d_j M= 1_t$, $0\le j\le 2$. This implies
$m_{11} = m_{22} = 1_t$ and $m_{12}\in \mathop{Ker}\left( \partial_t \right)$.
Two elements $M^1$ ,$M^2\in N_2\left( A,\cat \right)$ belong to the same class
if and only if exists $M\in N_2\left( A,\cat \right)$ such that $b^3M= \left(
s_0\left( 1_t \right), s_0\left( 1_t \right), M^1,M^2 \right)$. Such $M$
necessarily has the form
$$
\mu_3\left( s_0\left( 1_t \right), M^2, e_t \right) = \left( 
\begin{array}{ccc}
	1_t & m^2_{12} &e_t \\
	& 1_t &e_t \\
	& & 1_t
\end{array}
\right) 
$$
and therefore $m^2_{12} = m^1_{12}$. Thus we see that $\pi_2 =
\mathop{Ker}\left( \partial_t \right)$ as a set. Let $M^1$, $M^2 \in
\pi_2$. Then 
$$
M :=  \mu_3\left(s_0\left( 1_t \right), M^2, m^1_{12} \right) = 
\left( 
\begin{array}{ccc}
	1_t & m^2_{12} & m^1_{12} \\
	& 1_t & \\
	&& 1_t
\end{array}
\right)
$$
is the preimage of $\left( s_0\left( 1_t \right), M^1, \varnothing, M^2
\right)$ under $b^3_2$. Therefore
$$
M^1 M^2= d_2 M = \left( 
\begin{array}{cc}
	1_t & m^2_{12}m^1_{12}\\
	& 1_t
\end{array}
\right).
$$
As $\mathop{Ker}\left( \partial_t \right)$ is a commutative group we see that
$\pi_2$ and $\mathop{Ker}\left( \partial_t \right)$ are isomorphic as groups. 

Since $N\left( A,\cat \right)$ is a $3$-coskeletal set all other homotopy groups
of $N\left( A,\cat \right)$ are trivial. Thus $N\left( A,\cat \right)$ is a
$2$-type. 

\begin{proposition}
	Let $\left( A,\cat \right)$ be a crossed module of monoids such that
	$N\left( A,\cat \right)$ is a Kan set. Then $\left( A,\cat \right)$ is a
	crossed module. 
\end{proposition}
\begin{proof}
	We have to show that $\cat$ is a groupoid and that for every $t\in
	\cat_0$ the monoid $A\left( t \right)$ is a group. Let $g\in \cat\left(
	s,t \right)$. Then $\left( g, 1_s,\varnothing  \right)$ and $\left( 
	\varnothing,1_t,  g
	\right)$ are elements of $\bigwedge^2_2 N\left( A,\cat \right)$ and
	$\bigwedge^2_0 N\left( A,\cat \right)$ respectively. Since
	$N\left( A,\cat \right)$ is a Kan simplicial set there exist their
	preimages  $M^1$ and
	$M^2$  in $N_2\left( A,\cat \right)$ under $b^2_2$ and $b^2_0$
	respectively. Then $1_s = d_2M^1 =g\partial\left( m^1_{12}
	\right)m^1_{22} $ and $1_t = d_2 M^2 = m^2_{11}\partial_t\left(
	m^2_{22}g
	\right)$, which shows that $g$ has left and right inverse. By the usual
	trick they are equal to each other. 

Now let $a \in A\left( t \right)$. We consider
$$
\left( 
\left( 
\begin{array}{ccc}
	1_t & e_t & a \\
	& 1_t & e_t\\
	&& 1_t
\end{array}
\right), 
\left( 
\begin{array}{ccc}
		1_t & e_t & e_t \\
	& 1_t & e_t\\
	&& 1_t
\end{array}
\right), \varnothing, 
\left( 
\begin{array}{ccc}
		1_t & e_t & e_t \\
	& 1_t & e_t\\
	&& 1_t
\end{array}
\right),
\left( 
\begin{array}{ccc}
	1_t & e_t & a \\
	& 1_t & e_t\\
	&& 1_t
\end{array}
\right)
\right), 
$$
which is an element of $\bigwedge^5_2 N\left( A,\cat \right)$. Since $N\left(
A,\cat \right)$ is a Kan simplicial set there exists $M$ in the preimage of this
element under $b^5_2$. This matrix necessarily has the form
$$
\left( 
\begin{array}{ccccc}
	1_t & e_t & a & m_{15}\\
	& 1_t & e_t & a\\
	&& 1_t & e_t\\
	&&& 1_t
\end{array}
\right).
$$
From the explicit form for $d_1 M$ and $d_3 M$
we see that $m_{15}$ is the inverse element to $a$ in $A\left( t \right)$. 
\end{proof}

\bibliography{nerv}
\bibliographystyle{amsalpha}

\end{document}